\documentclass[11pt]{article}
\usepackage{amssymb,amsmath,amsthm}
\usepackage{a4}
\theoremstyle{plain}
\newtheorem{theorem}{Theorem}[section]
\newtheorem{proposition}[theorem]{Proposition}
\newtheorem{lemma}[theorem]{Lemma}
\newtheorem*{lemma*}{Auxiliary Lemma}

\newtheorem{conjecture}[theorem]{Conjecture}
\newtheorem{problem}[theorem]{Problem}

\theoremstyle{definition}

\newtheorem*{acknowledgement}{Acknowledgement}

\DeclareMathOperator{\st}{st}
\DeclareMathOperator{\cst}{\overline{st}}
\DeclareMathOperator{\lk}{lk}
\DeclareMathOperator{\antist}{ast}
\DeclareMathOperator{\Z}{\mathbb{Z}}
\DeclareMathOperator{\g}{g}
\DeclareMathOperator{\f}{f}
\DeclareMathOperator{\h}{h}
\title{Remarks on Missing Faces and Generalized Lower Bounds on Face Numbers}
\author{Eran Nevo\footnote{Department of Mathematics, Cornell
University, Ithaca USA, E-mail address:
eranevo@math.cornell.edu. Research partially supported by an NSF Award DMS-0757828.}}
\begin{document}
\maketitle
\begin{center}
\emph{Dedicated to Anders Bj\"{o}rner on the occasion of his 60th birthday.}
\end{center}
\begin{abstract}
We consider simplicial polytopes,
and more general simplicial complexes, without missing faces above a fixed dimension.
Sharp analogues of McMullen's generalized lower bounds, and of Barnette's lower bounds, are conjectured for these families of complexes.
Some partial results on these conjectures are presented.
\end{abstract}
\section{Introduction}\label{sec:Intro}
For simplicial polytopes, McMullen's generalized lower bounds on face numbers were proved by Stanley, conveniently phrased as nonnegativity of the corresponding $g$-vector \cite{Stanley-gThm}. As the matrix that sends the $g$-vector to its $f$-vector has nonnegative entries, Stanley's result immediately implies Barnette's lower bound theorem for simplicial polytopes. In turn, Barnette's result immediately implies that the simplex minimizes all face numbers among simplicial polytopes with the same dimension.

We will phrase conjectures analogous to these three results, depending on a new parameter, namely the maximal size of a missing face, i.e. the maximal size of a
minimal non face with respect to inclusion. This gives a hierarchy of conjectures on lower bounds on face numbers, interpolating between the generalized lower bound conjecture for simplicial spheres \cite{McMullen-g-conj} and Gal's conjecture for flag spheres \cite{Gal}.
We will now work our way to these three conjectures, from weakest to strongest.

It is well known, and easy to prove, that among all simplicial complexes with a nonzero (reduced) $d$-homology, the boundary of the $(d+1)$-simplex minimizes all face numbers. Similarly, Meshulam proved that among all flag complexes with a nonzero (reduced) $d$-homology, the boundary of the $(d+1)$-dimensional crosspolytope minimizes all face numbers \cite{Meshulam-DomAndHom}. (Recently  Athanasiadis \cite{Athanasiadis} proved that for the subfamily of flag homology $d$-spheres the $h$-vector is minimized by the boundary of the $(d+1)$-dimensional crosspolytope, hence so are the face numbers.)

For this minimization problem we can clearly assume that the complexes are $d$-dimensional, as restricting to the $d$-skeleton cannot make the $d$-th homology vanish and cannot increase the face numbers.

We find it natural to view these two families of simplicial complexes as extreme cases of the following families. Let $\mathcal{C}(i,d)$ be the family of $d$-dimensional simplicial complexes with a nonzero reduced $d$-homology (if $\Delta\in \mathcal{C}(i,d)$ then $\tilde{H}_d(\Delta;\mathbb{Z})\neq 0$) and with no missing faces of dimension $>i$. ($F$ is a \emph{missing face} of $\Delta$ if its boundary $\partial F\subseteq \Delta$ and $F\notin \Delta$. Its dimension is $|F|-1$.) Thus, $\mathcal{C}(1,d)$ are the flag $d$-complexes with nonzero $d$-th homology and $\mathcal{C}(d+1,d)$ are all the $d$-complexes with a nonzero $d$-th homology. (Clearly if $i>d+1$ then $\mathcal{C}(i,d)=\mathcal{C}(d+1,d)$.)
Denote by $\f_i(\Delta)$ the number of $i$-dimensional faces in the complex $\Delta$.

Let $d\geq 0, 0<i$ be integers. Then there exist unique integers $q\geq 0, 1\leq r\leq i$ such that $d+1=qi+r$.
(Note that the range $1\leq r\leq i$ is unusual. It will simplify the writing later on.)
Let
$$S(i,d):= \partial\sigma^i*...*\partial\sigma^i*\partial\sigma^r,$$
where $\partial\sigma^i$, the boundary of the $i$-simplex, appears $q$ times in this join. Then $S(i,d)$ is a $d$-dimensional simplicial sphere.
Inspired by Meshulam \cite{Meshulam-DomAndHom}, we prove that it has the following extremal properties:
\begin{theorem}\label{thm:missingLBT}
Let $d\geq 0, 0<i\leq d+1$ be integers. Write $d+1=qi+r$ where $1\leq r\leq i$ and $q,r$ are integers.
Let $\Delta \in \mathcal{C}(i,d)$. Then:

(a) If $i$ divides $d+1$ then $\f_j(\Delta)\geq \f_j(S(i,d))$ for every $j$.

(b) For any $i$, $\f_0(\Delta)\geq \f_0(S(i,d))$.

(c) For any $i$, $\f_j(\Delta)\geq \f_j(S(i,d))$ for every $1\leq j\leq r$.

(d) If $i$ divides $d+1$ and $\f_j(\Delta)=\f_j(S(i,d))$ for every $j$ then $\Delta=S(i,d)$.
\end{theorem}
The cases $i=d+1$ and $i=1$ recover the two known results mentioned above.
We write (b) and (c) separately on purpose, as (b) will play a special role.
The condition $i\mid (d+1)$ in part $(a)$ seems to be an artifact of the proof.
\begin{conjecture}\label{conj:missingLBC}
Let $d\geq 0$ and $0<i\leq d+1$ be integers.
Let $\Delta \in \mathcal{C}(i,d)$. Then
$\f_j(\Delta)\geq \f_j(S(i,d))$ for every $j$.

Moreover, if equality is attained for every $j$ then $\Delta=S(i,d)$.
\end{conjecture}

A refined question is to give lower bounds on face numbers of complexes in $\mathcal{C}(i,d)$  with a given number of vertices. The answer to this question for simplicial polytopes is a well known result by Barnette, often referred to as `The lower bound theorem' \cite{Barnette-LBT}. Barnette later showed that these lower bounds hold for all triangulated manifolds \cite{BarnetteManifolds}. Kalai showed they hold for all homology spheres, and more general complexes, and characterized the case of equality \cite{Kalai-LBT}.

To state these results, we define stacked polytopes and homology spheres.
 A \emph{stacking} is the operation of adding a pyramid over a facet of a given simplicial polytope.
A polytope is \emph{stacked} if it can be obtained from a simplex by repeating the stacking operation finitely many times. Let $Sk(d,n)$ be the boundary complex of a stacked $(d+1)$-polytope with $n$ vertices. While the combinatorial type of $Sk(d,n)$ is not unique, its face numbers are determined.
Next, a $d$-dimensional complex $\Delta$ is a
\emph{homology sphere} if for every face $F$ in $\Delta$ (including the empty set), and for every $0\leq j$, there is an isomorphism of reduced homology groups $\tilde{H}_j(\lk(F,\Delta);\Z)\cong \tilde{H}_j(S^{\dim(\Delta)-|F|};\Z)$ where $S^m$ denotes the $m$-dimensional sphere, $\Z$ the integers and $\lk(F,\Delta)$ is the link of $F$ in $\Delta$. In particular, the boundary complex of a simplicial polytope is a homology sphere; however there are many non-polytopal examples of homology spheres, e.g. \cite{Kalai-manyspheres}. The following is the lower bound theorem (LBT):
\begin{theorem}\label{thm:stacked,Kalai}(\cite{BarnetteManifolds,Barnette-LBT} and \cite{Kalai-LBT})
Let $d\geq 3$, and let $\Delta$ be the boundary complex of a simplicial $(d+1)$-polytope, or more generally a homology $d$-sphere, with $n$ vertices. Then
$\f_j(\Delta)\geq \f_j(Sk(d,n))$ for every $j$.
If equality holds for some $j\geq 1$ then $\Delta$ is combinatorially isomorphic to some $Sk(d,n)$.
\end{theorem}
We now seek an analogue of this result when an upper bound on the dimension of missing faces is specified. Let $\mathcal{HS}(i,d,n)$ be the family of $d$-dimensional homology spheres with $n$ vertices and without missing faces of dimension $>i$.
Let $d\geq 0, 0<i$ and $d+1=qi+r$ as before ($1\leq r\leq i$). If $\mathcal{HS}(i,d,n)\neq \emptyset$, then by Theorem \ref{thm:missingLBT}(b), $n\geq q(i+1)+(r+1)$. Hence the following definition makes sense:  $S(i,d,n):= \partial\sigma^i*...*\partial\sigma^i*Sk(r-1, n-q(i+1))$, where $\partial\sigma^i$ appears $q$ times in this join. This is possible unless $r=1$ and $n>q(i+1)+2$. In the later case define $S(i,d,n)= \partial\sigma^i*...*\partial\sigma^i*Sk(i, n-(q-1)(i+1))$, where $\partial\sigma^i$ appears $q-1$ times in this join.
In any case, $S(i,d,n)\in \mathcal{HS}(i,d,n)$.
\begin{conjecture}\label{conj:missingLBCf_0}
If $\Delta\in \mathcal{HS}(i,d,n)$ then $\f_j(\Delta)\geq \f_j(S(i,d,n))$ for every $j$.
\end{conjecture}
Clearly Conjecture \ref{conj:missingLBCf_0} implies Conjecture \ref{conj:missingLBC} restricted to homology spheres.
For $i=d+1$ and $i=d$ the assertion of the conjecture is the LBT, Theorem \ref{thm:stacked,Kalai}. For $i=d-1$ the conjecture holds if $\Delta$ is the boundary of a simplicial polytope, and follows from the celebrated $g$-theorem \cite{Billera-Lee, Stanley-gThm}. Here Stanley's result is used. Surprisingly, in this case equalities for all $j$'s in the conjecture imply that $\Delta=S(d-1,d,n)$. This follows from recent results in \cite{Nevo-Novinsky}.
For $i=1$, the conjectured lower bounds (for flag homology spheres) would follow from Gal's conjecture on the $\gamma$-polynomial \cite[Conjecture 2.1.7]{Gal}.
In this case there are many examples of equalities for all $j$'s in the conjecture.

We now relate the $g$-theorem and Gal's conjecture, by defining new `$g$-vectors' suitable for the families of $d$-dimensional homology spheres without missing faces of dimension $>i$, denoted by $\mathcal{HS}(i,d)$.

Let $d\geq 0$ and $i>0$ be integers, and let $q\geq 0, 1\leq r\leq i$ be the unique integers such that $d+1=qi+r$.
For such $d$ and $i$ define the polynomial
$$P_{d,i}(t):=(1+t+...+t^i)^q(1+t+...+t^r).$$
It is symmetric as a multiplication of symmetric polynomials.
Further, denote $P_{-1,i}:=1$ (a constant polynomial).
Define the ordered set of polynomials
$$B_{d,i}:=(P_{d,i}(t), tP_{d-2,i}(t), t^2P_{d-4,i}(t),...,t^{\lfloor\frac{d+1}{2}\rfloor}P_{d-2\lfloor\frac{d+1}{2}\rfloor,i}(t)).$$
Note that $B_{d,i}$ is a basis for the space of symmetric polynomials of degree at most $d+1$ and axis of symmetry at `degree' $\frac{d+1}{2}$ (over the rationales, say).
For a symmetric polynomial $\h(t)$ in this space, let $\g^{(d,i)}(\h(t))=(\g^{(d,i)}_0,...,\g^{(d,i)}_{\lfloor\frac{d+1}{2}\rfloor})$ be the vector of coefficients in the expansion of $\h(t)$ in the basis $B_{d,i}$.

Given an $f$-\emph{vector} $(\f_{-1},\f_0,...,\f_{d})$, its corresponding $h$-\emph{vector} $\h =(\h_0,...,\h_{d+1})$ is defined by $\sum_{0\leq i\leq d+1}\h_i x^{d+1-i}= \sum_{0\leq i\leq d+1}\f_{i-1}(x-1)^{d+1-i}$. Clearly, the $h$-vector  carries the same combinatorial information as the $f$-vector. For $\Delta\in \mathcal{HS}(i,d)$ the Dehn-Sommerville relations state that the $h$-polynomial of $\Delta$, $\h_{\Delta}(t)=\h_0+\h_1t+...+h_{d+1}t^{d+1}$, is symmetric (e.g. \cite{StanleyGreenBook}).
Define
$$\g^{(i)}(\Delta):=\g^{(d,i)}(\h_{\Delta}(t)).$$
Clearly, $\g^{(i)}(\Delta)$ is an integer vector.
\begin{conjecture}\label{conj:missing-g-conj}
If $\Delta\in \mathcal{HS}(i,d)$ then $\g^{(i)}(\Delta)\geq 0$ (componentwise).
\end{conjecture}
Note that $\g^{(i)}_j(S(i,d,n))=0$ for any $2\leq j\leq \lfloor\frac{d+1}{2}\rfloor$. Further,
recall that $\sum_{0\leq i\leq d}\h_i (y+1)^{d-i}= \sum_{0\leq i\leq d}\f_{i-1}y^{d-i}$.
Thus,
Conjecture \ref{conj:missing-g-conj} implies Conjecture \ref{conj:missingLBCf_0}. Note that for $d\leq 4$ these two conjectures are equivalent.

The classical $g$-vector of a $d$-dimensional homology sphere $\Delta$ is $\g(\Delta)=\g^{(d+1)}(\Delta)$, which the $g$-conjecture asserts to be nonnegative, and its Gal polynomial is $\gamma(\Delta)=\g^{(1)}(\Delta)$, which is conjectured to be nonnegative in the flag case.
Note that $\mathcal{HS}(i,d)\subseteq \mathcal{HS}(i+1,d)$. As one may expect, Conjecture \ref{conj:missing-g-conj} set a hierarchy, in the following sense.
\begin{proposition}\label{prop:g-hierarchy}
Let $\h(t)$ be a polynomial in the space of symmetric polynomials of degree at most $d+1$ and axis of symmetry at `degree' $\frac{d+1}{2}$.
If $\g^{(d,i)}(\h(t))\geq 0$ then $\g^{(d,i+1)}(\h(t))\geq 0$.
\end{proposition}
For a different relation between face numbers and numbers of missing faces we refer to Nagel \cite{Nagel}.

In Section \ref{sec:prove} we prove Theorem \ref{thm:missingLBT} and discuss Conjecture \ref{conj:missingLBC}. In Section \ref{sec:discuss} we give evidence for Conjecture \ref{conj:missingLBCf_0} and reduce the flag case in Conjecture \ref{conj:missingLBCf_0} to the inequality for the number of edges.
In Section \ref{sec:hierarchy} we prove Proposition \ref{prop:g-hierarchy} and discuss Conjecture \ref{conj:missing-g-conj}, including its first open case, namely $\mathcal{HS}(2,4)$.

\section{Around $S(i,d)$}\label{sec:prove}
We start this section with a proof of Theorem \ref{thm:missingLBT} and end it with a few comments on Conjecture \ref{conj:missingLBC}.

Let $\Delta$ be a simplicial complex and $F\in \Delta$ a face.
The \emph{link} of $F$ in $\Delta$ is $\lk(F)=\lk(F,\Delta)=\{T\in \Delta: \ T\cap F=\emptyset, \ T\cup F\in \Delta\}$, its \emph{closed star} is $\cst(F)=\cst(F,\Delta)=\{T\in \Delta: \ T\cup F\in \Delta\}$,
its (closed) \emph{antistar} is $\antist(F)=\antist(F,\Delta)=\{T\in \Delta: \ F\nsubseteq T \}$; they are simplicial complexes as well. The (open) \emph{star} of $F$ in $\Delta$ is the collection of sets $\st(F)=\st(F,\Delta)=\{T\in \Delta: \ F\subseteq T\}$.
\begin{lemma}\label{lem:H_d(ast)=0}
Theorem \ref{thm:missingLBT} follows from the special case of it where in addition to $\Delta \in \mathcal{C}(i,d)$ we assume that for every $F\in \Delta$ with $\dim(F)\leq i$, $\tilde{H}_d(\antist(F,\Delta);\mathbb{Z})=0$.
\end{lemma}
\begin{proof}
Notice that $F$ is the only missing face in $\antist(F,\Delta)$ which is not missing in $\Delta$. As $\dim(F)\leq i$,
$\antist(F,\Delta)$ has no missing faces of dimension $>i$. As $\f_j(\Delta)\geq \f_j(\antist(F,\Delta))$ for every $j$, if $\tilde{H}_d(\antist(F,\Delta);\mathbb{Z})\neq 0$ then Theorem \ref{thm:missingLBT} would follow by induction on the sum of face numbers, $\f=\sum_j \f_j$.
\end{proof}

\begin{lemma}\label{lem:H(lk)_is_not_0}
Let $\Delta \in \mathcal{C}(i,d)$. If $F\in \Delta$ with $\tilde{H}_d(\antist(F,\Delta);\mathbb{Z})=0$ then $\tilde{H}_{d-|F|}(\lk(F,\Delta);\mathbb{Z})\neq 0$.
\end{lemma}
\begin{proof}
Consider the Mayer-Vietoris long exact sequence for the union $\Delta=\antist(F,\Delta)\cup_{\partial F* \lk(F,\Delta)}\cst(F,\Delta)$. As $\cst(F,\Delta)$ is contractible, $\tilde{H}_d(\Delta,\mathbb{Z})$ injects into $\tilde{H}_{d-1}(\partial F* \lk(F,\Delta))$. As $\partial F$ is a sphere of dimension $|F|-2$, the K\"{u}nneth formula implies that $\tilde{H}_{d-|F|}(\lk(F,\Delta);\mathbb{Z})\neq 0$.
\end{proof}
\begin{lemma}\label{lem:missing_faces(lk)}
Let $\Delta \in \mathcal{C}(i,d)$. If $F\in \Delta$ then there are no missing faces of dimension $>i$ in $\lk(F,\Delta)$.
\end{lemma}
\begin{proof}
Assume by contradiction that $T$ is missing in $\lk(F,\Delta)$, of dimension $>i$. Then $T\in \Delta$ and $F\neq \emptyset$. Let $v\in F$. Then $\partial(\{v\}\cup T)\subseteq \Delta$, hence $\{v\}\cup T\in \Delta$ as $\Delta$ has no missing faces of dimension $>i$. Similarly, if $v\neq u\in F$, we already argued that $\partial(\{v,u\}\cup T)\subseteq \Delta$, hence $\{v,u\}\cup T\in \Delta$, and inductively we conclude that $T\cup F\in \Delta$, contradicting the assumption that $T$ is missing in $\lk(F,\Delta)$.  \end{proof}
\begin{lemma}\label{lem:f-recurrence}
For all integers $d\geq 0$, $0<i\leq d+1$ and $j$,
$$\f_j(S(i,d))=\f_j(S(i,d-1))+\f_{j-1}(S(i,d-1))+\f_{j-r}(S(i,d-r))$$
where $d+1=qi+r$ as in the definition of $S(i,d)$ ($1\leq r\leq i$).
\end{lemma}
\begin{proof}
In the definition of $S(i,d)$ let the set of vertices of $\sigma^r$ be $F\uplus \{v\}$. Then the number of $j$-faces in $S(i,d)$ containing $v$ is $\f_{j-1}(S(i,d-1))$, of those containing $F$ is $\f_{j-r}(S(i,d-r))$ and of those containing neither $v$ nor $F$ is $\f_j(S(i,d-1))$. The assertion thus follows.
\end{proof}
\begin{proof}[Proof of Theorem \ref{thm:missingLBT}(a)] This is the case $r=i$.
The case $d=0$ is trivial: here $i=r=1$, $q=0$ and $S(1,0)$ consists of two points. We proceed by double induction on $d$ and on $\f=\sum_j \f_j(\Delta)$. By induction on $\f$ we may assume that for every $v\in \Delta$, $\tilde{H}_d(\antist(v,\Delta);\mathbb{Z})=0$; see Lemma \ref{lem:H_d(ast)=0}. By Lemmata \ref{lem:H(lk)_is_not_0}, \ref{lem:missing_faces(lk)} and the induction on $d$, we obtain the following inequality for every $j$ and every $v\in \Delta$:
\begin{eqnarray}\label{eq:ineqThm(a)}
\f_j(\Delta)=\f_j(\st(v))+\f_j(\lk(v))+\f_j(\antist(v)\setminus \lk(v))=\nonumber\\
\f_{j-1}(\lk(v))+\f_j(\lk(v))+\f_j(\antist(v)\setminus \lk(v))\geq \nonumber\\ \f_{j-1}(S(i,d-1))+\f_j( S(i,d-1))+\f_j(\antist(v)\setminus \lk(v)).
\end{eqnarray}
For any face $F\in \antist(v)\setminus \lk(v)$, $\f_j(\antist(v)\setminus \lk(v))\geq \f_j(\st(F))=\f_{j-|F|}(\lk(F))$.

If $|F|=l\leq r$ then by Lemma \ref{lem:f-recurrence}, and the following observation, Theorem \ref{thm:missingLBT}(a) follows: notice that $\f_{j-l}(S(i,d-l))\geq \f_{j-r}(S(i,d-r))$, as for a fixed subset $T$ of the vertices of $\sigma^r$ of size $r-l$, the map $A\mapsto A\cup T$ is an injection from the $(j-r)$-faces of $S(i,d-r)$ into the
$(j-l)$-faces of $S(i,d-l)$.

When $r=i$ a minimal face $F\in \antist(v)\setminus \lk(v)$ is of size $\leq r$, as otherwise $F\cup \{v\}$ would be a missing face of $\Delta$ of dimension $>i$, a contradiction.
\end{proof}
To prove Theorem \ref{thm:missingLBT}(b), we use the following result of Bj\"{o}rner et. al. \cite{Bjorner-NoteAlexanderDuality}, as in Meshulam \cite[Remark in Section 2]{Meshulam-DomAndHom}.
\begin{lemma}(\cite[Theorem 2]{Bjorner-NoteAlexanderDuality})\label{lem:Bjorner-NoteAlexanderDuality}
Let $\Delta$ be a simplicial complex on a vertex set $V$ of size $n$, and $C$ be its set of missing faces. Let $\Gamma$ be the simplicial complex on the vertex set $C$, whose nonempty faces are the $F\subseteq C$ such that the union $\cup_{u\in F}\{u\}\neq V$. If $V\notin \Delta$ then $\tilde{H}_j(\Delta;\Z)\cong\tilde{H}^{n-j-3}(\Gamma;\Z)$ for every $j$. ($\tilde{H}^j(\cdot;\Z)$ denotes $j$-th cohomology.)
\end{lemma}
\begin{proof}[Proof of Theorem \ref{thm:missingLBT}(b)]
As $\Delta\in \mathcal{C}(i,d)$, $\Delta$ is not a simplex and Lemma \ref{lem:Bjorner-NoteAlexanderDuality} applies. Thus $\tilde{H}^{n-d-3}(\Gamma)\neq 0$. In particular, $\Gamma$ does not have a complete $(n-d-2)$-skeleton, that is, there exists $A\subseteq C$, $|A|=n-d-1$ and $\cup_{a\in A}\{a\}= V$. Each element in $C$ has size $\leq i+1$, therefore $(i+1)(n-d-1)\geq n$, equivalently $n\geq \frac{(d+1)(i+1)}{i}=d+1+(q+\frac{r}{i})$. Thus $\f_0(\Delta)\geq d+1+q+1=\f_0(S(i,d))$.
\end{proof}
\begin{proof}[Proof of Theorem \ref{thm:missingLBT}(c)]
Note that the proof of Theorem \ref{thm:missingLBT}(a) would fail for $r\neq i$ only if for every vertex $v\in \Delta$ all the minimal faces in $\antist(v)\setminus \lk(v)$ have dimension $\geq r$. This happens only if all the missing faces in $\Delta$ have dimension $>r$, in which case $\Delta$ has a complete $r$-skeleton, and by Theorem \ref{thm:missingLBT}(b), for any $1\leq j\leq r$,
$\f_j(\Delta)\geq \binom{\f_0(S(i,d))}{j+1}\geq\f_j(S(i,d))$.
\end{proof}
\begin{proof}[Proof of Theorem \ref{thm:missingLBT}(d)]
This follows from the proof of Theorem \ref{thm:missingLBT}(a).
Indeed, for a vertex $v$ and a face $F\in \antist(v)$, (\ref{eq:ineqThm(a)}) and the argument following it imply $\st(F)=\antist(v)\setminus \lk(v)$, $|F|=r=i$, and by induction on the dimension $\lk(F)=S(i,d-i)$. Thus, $\Delta=\partial(F\cup \{v\})*\lk(F)=S(i,d)$.
\end{proof}
In Conjecture \ref{conj:missingLBC}, the following cases are of
particular interest: The case $i=d-1$ includes prime spheres. Recall that a triangulated sphere is \emph{prime} if it has no missing facets and it is not the boundary of a simplex. The conjecture
implies that the join of boundaries of a triangle and a $(d-1)$-simplex minimizes the face numbers among prime $d$-spheres for $d>2$ (and the octahedron minimizes the face numbers among prime $2$-spheres, which coincides with the proven $i=1$ case).
The case $i=d$ implies that among all simplicial $d$-spheres different from the boundary of the $(d+1)$-simplex, the bipyramid minimizes the face numbers. This case easily follows from Theorem \ref{thm:stacked,Kalai}.

\section{Around $S(i,d,n)$}\label{sec:discuss}
We discuss the cases $i=d+1,d,d-1,1$ in order, indicating by bullets the cases where Conjecture \ref{conj:missingLBCf_0} is known. We end this section with a reduction of the case $i=1$ of this conjecture to the inequality for $\f_1$.\\
$\bullet$
The cases $i\in\{d+1,d\}$ of Conjecture \ref{conj:missingLBCf_0} follow from Theorem \ref{thm:stacked,Kalai}.
\\
$\bullet$
The case $i=d-1$ of this conjecture holds for $d\leq 4$ and for $\Delta$ the boundary of a simplicial polytope.
\\
In these cases the $g$-vector of $\Delta$ is nonnegative: by a rigidity argument for $d\leq 4$ \cite{Kalai-LBT,Gromov}, and by the necessity part of the $g$-theorem for boundaries of simplicial polytopes \cite{Stanley-gThm}. Also, as $i<d$, $\Delta$ is not stacked, hence $\g_2(\Delta)\geq 1$ by the equality case in Theorem \ref{thm:stacked,Kalai}. But $\g_2(S(d-1,d,n))=1$ and $\g_j(S(d-1,d,n))=0$ for $j>2$, thus $\f_j(\Delta)\geq \f_j(S(d-1,d,n))$ for every $j$, by the same argument that showed how to conclude Conjecture \ref{conj:missingLBCf_0} from Conjecture \ref{conj:missing-g-conj} (see Introduction).

In \cite{Nevo-Novinsky} Novinsky and the author showed that if $\Delta$ is a prime $d$-dimensional homology sphere with $\g_2(\Delta)=1$ then either $\Delta=S(d-1,d,n)$ (if $n>d+3$) or $\Delta=S(i,d)$ for $\frac{d+1}{2}\leq i\leq d-1$ (if $n=d+3$). In the later case $S(d-1,d)=S(d-1,d,d+3)$ has the smallest $g$-vector. Thus, among boundaries of prime polytopes, $S(d-1,d,n)$ is the \emph{unique} example with equalities in Conjecture \ref{conj:missingLBCf_0} for all $j$'s.

Unlike the cases $i\in\{d+1,d\}$ where equality for some $\f_j$ ($j>0$) implied equalities for all $\f_j$'s, here equality for $\f_1$ holds for all
$\Delta=S(i,d)$, $\frac{d+1}{2}\leq i\leq d-1$ (where $n=d+3$), but not for $\f_2$ for example.


An intermediate problem, stronger than Conjecture \ref{conj:missingLBCf_0} and weaker than Conjecture \ref{conj:missing-g-conj} is the following.
\begin{problem}
Show that for any $i,d,n$ and $\Delta\in \mathcal{HS}(i,d,n)$, $\h(\Delta)\geq \h(S(i,d,n))$ componentwise.
In particular, $\h(\Delta)\geq \h(S(i,d))$ follows.
\end{problem}
For $\Delta\in \mathcal{HS}(1,d,n)$, and more generally for $\Delta$ a doubly Cohen-Macaulay flag $d$-complex, Athanasiadis very recently proved that $\h(\Delta)\geq \h(S(1,d))$ \cite{Athanasiadis}.

We now discuss the interesting case $i=1$ (i.e. flag homology spheres) in more details.
\\
$\bullet$
Conjecture \ref{conj:missingLBCf_0} for flag homology spheres of dimension $\leq 4$ holds.
\\
Gal showed nonnegativity of the $\gamma$-polynomial in this case \cite{Gal}, based on the Davis-Okun proof \cite{Davis-Okun} of the $3$-dimensional case of Charney-Davis conjecture for homology flag spheres \cite[Conjecture D]{Charney-Davis}.
\\
$\bullet$
There is no conjecture for a characterization of the extremal cases when $i=1$.
\\
Here are some examples of extremal cases, i.e.
of $\Delta\in \mathcal{HS}(1,d,n)$ with $\f_j(\Delta)= \f_j(S(1,d,n))$ for any $j$. Note that equalities hold for any $2$-sphere.
Note further that subdividing edges of the octahedron always yield a flag $2$-sphere, and now subdivide edges until a sphere with $n$ vertices is obtained. Clearly, if equalities hold for $\Delta\in \mathcal{HS}(1,d,n)$, then equalities hold for its suspension $\Sigma(\Delta)\in \mathcal{HS}(1,d+1,n+2)$. Thus, more examples with equalities are obtained by repeated suspension of the above $2$-spheres.

There are more equality cases, not based on suspension.
Contracting an edge in $\Delta$ which is not contained in an induced $4$-cycle results in a flag homology sphere $\Delta'$ with one vertex less. If the link of such edge is the boundary of a crosspolytope (called also octahedral sphere) then it is easy to check that $\f_j(\Delta)-\f_j(S(1,d,\f_0(\Delta)))= \f_j(\Delta')-\f_j(S(1,d,\f_0(\Delta')))$. Start with an octahedral sphere and repeat performing the inverse of this operation. This produces new examples with equalities for all the $\f_j$'s, which are \emph{not} suspensions, in any dimension.

Barnette's proof of the LBT \cite{Barnette-LBT} made use of a reduction to the inequality for $\f_1$, known as the McMullen-Perles-Walkup reduction (MPW). The idea there is to double count the pairs ``a vertex in a $k$-face". This idea works here as well, reducing the case $i=1$ in Conjecture \ref{conj:missingLBCf_0} to the inequality for $\f_1$.
\begin{proposition}\label{prop:MPWclique}
If $\f_1(\Delta)\geq \f_1(S(1,d,n))$ holds for all $d,n$ and $\Delta\in \mathcal{HS}(1,d,n)$, then $\f_k(\Delta)\geq \f_k(S(1,d,n))$ holds for all $d,n,k$ and $\Delta\in \mathcal{HS}(1,d,n)$.
\end{proposition}
\begin{proof}
To obtain bounds of the form $\f_k\geq a_{k,d}\f_0 - b_{k,d}$, where
$a_{k,d}$ and $b_{k,d}$ are functions of $k$ and $d$,
the MPW double counting gives, summing over all vertices,
\begin{eqnarray*}
(k+1)\f_k(\Delta) = \sum_v \f_{k-1}(\lk(v,\Delta))\geq&&\\
\sum_v [a_{k-1,d-1}\f_0(\lk(v,\Delta))-b_{k-1,d-1}] =&&\\
a_{k-1,d-1}2\f_1(\Delta)-b_{k-1,d-1}\f_0(\Delta) \geq&&\\
2[(2d-1)\f_0(\Delta)-2(d+1)(d-1)]a_{k-1,d-1}-b_{k-1,d-1}\f_0(\Delta)
.&&
\end{eqnarray*}
We equate
$$(\mbox{k+1})a_{k,d}=2(2d-1)a_{k-1,d-1}-b_{k-1,d-1}$$ and $$(k+1)b_{k,d}=4(d+1)(d-1)a_{k-1,d-1},$$ with $a_{1,d}=2d-1$ and $b_{1,d}=2(d+1)(d-1)$.

We need to check now that the numbers $a_{k,d}$ and $b_{k,d}$ defined by the equations $\f_k(S(1,d,n))=a_{k,d}n-b_{k,d}$ for all $n$, indeed satisfy the above recurrence. We leave the case $k=d$ to the reader, and show the computation for $1<k<d$. It is based on two simple identities for binomial coefficients, namely $s\binom{t}{s}=t\binom{t-1}{s-1}$ and
$\binom{t+1}{s}=\binom{t}{s}+\binom{t}{s-1}$. Here are the details.
Note that
\begin{eqnarray*}
\f_k\left(S\left(1,d,n\right)\right)= \binom{d-1}{k+1}2^{k+1}+\binom{d-1}{k}2^k\left(n-2\left(d-1\right)\right)+\binom{d-1}{k-1}2^{k-1}\left(n-2\left(d-1\right)\right)&=&\\
2^{k-1}\left[\binom{d-1}{k-1}+2\binom{d-1}{k}\right]n\ -\ 2^k\left(\left(d-1\right)\left[\binom{d-1}{k-1}+2\binom{d-1}{k}\right]-2\binom{d-1}{k+1}\right) &=&\\
2^{k-1}\left[\binom{d-1}{k}+\binom{d}{k}\right]n\ -\ 2^{k}\left(\left(d+1\right)\left[\binom{d-1}{k}+\binom{d}{k}\right]-2\binom{d+1}{k+1}\right).
\end{eqnarray*}
We now verify the recurrence for $a_{k,d}$.
\begin{eqnarray*}
\frac{1}{k+1}\left[2\left(2d-1\right)a_{k-1,d-1}-b_{k-1,d-1}\right] &=&\\
\frac{1}{k+1}\left[2\left(2d-1\right)2^{k-2}\left[\binom{d-2}{k-1}+\binom{d-1}{k-1}\right]- 2^{k-1}\left(d\left[\binom{d-2}{k-1}+\binom{d-1}{k-1}\right]-2\binom{d}{k}\right)\right]&=&\\
\frac{2^{k-1}}{k+1}\left[\left(d-1\right)\left[\binom{d-2}{k-1}+\binom{d-1}{k-1}\right]+2\binom{d}{k}\right]&=&\\
\frac{2^{k-1}}{k+1}\left[k\binom{d-1}{k}+k\binom{d}{k}-\binom{d-1}{k-1} +2\binom{d}{k}\right]&=&\\
\frac{2^{k-1}}{k+1}\left[\left(k+1\right)\binom{d-1}{k}+\left(k+1\right)\binom{d}{k}-\binom{d-1}{k} -\binom{d-1}{k-1}- \binom{d}{k} +2\binom{d}{k}\right]&=&\\
2^{k-1}\left[\binom{d-1}{k}+\binom{d}{k}\right]&=&a_{k,d}.
\end{eqnarray*}
We now verify the recurrence for $b_{k,d}$.
\begin{eqnarray*}
\frac{1}{k+1}\left[4\left(d+1\right)\left(d-1\right)a_{k-1,d-1}\right]&=&\\
\frac{1}{k+1}\left[4\left(d+1\right)\left(d-1\right)2^{k-2}\left[\binom{d-2}{k-1}+\binom{d-1}{k-1}\right]\right]&=&\\
\frac{2^k\left(d+1\right)}{k+1}\left[k\binom{d-1}{k}+k\binom{d}{k}-\binom{d-1}{k-1}\right]&=&\\
2^k\left(d+1\right)\left[\binom{d-1}{k}+\binom{d}{k}\right] + \frac{2^k\left(d+1\right)}{k+1}\left[-\binom{d-1}{k}-  \binom{d}{k}-\binom{d-1}{k-1}\right]&=&\\
2^{k}\left(\left(d+1\right)\left[\binom{d-1}{k}+\binom{d}{k}\right]- \frac{2\left(d+1\right)}{k+1}\binom{d}{k}\right)&=& b_{k,d}.
\end{eqnarray*}
This completes the proof.
\end{proof}

\section{Around $\g^{(d,i)}$}\label{sec:hierarchy}
\begin{proof}[Proof of Proposition \ref{prop:g-hierarchy}]
We need to show that the polynomial $P_{d,i}(t)$ is a nonnegative integer linear combination of the elements of the basis $B_{d,i+1}$, for all integers $0<i$ and $0\leq d$.

For a nonnegative integer $m$ let $S_m$ denote the $\mathbb{Z}$-module of symmetric polynomials over $\mathbb{Z}$ of degree at most $m$ with axis of symmetry at $\frac{m}{2}$.
Let $0<b\leq a$ be integers. Note that
\begin{equation}\label{eq:poly}
(1+t+...+t^b)(1+t+...+t^a)=(1+t+...+t^{b-1})(1+t+...+t^{a+1}) + t^b(1+t+...+t^{a-b}),
\end{equation}
expressing one polynomial in $S_{a+b}$ as a nonnegative integer linear combination of two other polynomials in $S_{a+b}$.

In the product $P_{d,i}(t)=(1+t+...+t^i)^q(1+t+...+t^r)$ apply equation (\ref{eq:poly}) to a pair of terms with degrees $a=i$ and $b=r$, thus expressing $P_{d,i}$ as a nonnegative integer linear combination of polynomials in $S_{d+1}$ where in each summand there is at most one term of the form $1+t+...+t^{k}$ with $0<k<i$.
Inductively, repeat this process with each summand, pairing two terms with degrees $a,b$ where $0<b\leq a=i$ as long as possible. The end result is that $P_{d,i}(t)$ is expressed as a nonnegative integer linear combination of polynomials of the form $t^c(1+t+...+t^{i+1})^{q'}(1+t+...+t^{r'})$ where $d+1=2c+q'(i+1)+r'$ and $0\leq r'\leq i$. In particular, all of these polynomials are in $B_{d,i+1}$.
\end{proof}
For two simplicial complexes, the missing faces of their join are the disjoint union of the missing faces of each of them. Thus, if $\Delta \in \mathcal{HS}(i,d)$ and $\Delta' \in \mathcal{HS}(i,d')$ then $\Delta*\Delta' \in \mathcal{HS}(i,d+d'+1)$. The following proposition shows that if Conjecture \ref{conj:missing-g-conj} holds for $\Delta$ and $\Delta'$ then it holds for their join too. For the family of all homology spheres, i.e. the case $i\geq d+d'+2$, this is clear, as well as for the flag case ($i=1$) which was verified by Gal \cite{Gal}.
\begin{proposition}\label{prop:g-join}
Let $h(t)\in S_{m+1}$ and $h'(t)\in S_{m'+1}$. If $\g^{(m,i)}(h(t)))$ and $\g^{(m',i)}(h'(t)))$ are nonnegative then $\g^{(m+m'+1,i)}(h(t)h'(t))$ is nonnegative.
\end{proposition}
\begin{proof}
We need to show that the polynomial $P_{d,i}(t)P_{d',i}(t)$ is a nonnegative integer linear combination of the elements of the basis $B_{d+d'+1,i}$, for all integers $0<i$ and $0\leq d,d'$.
Let $d+1=qi+r$ and $d'+1=q'i+r'$ where $q,q',r,r'$ are the unique nonnegative integers such that $1\leq r,r'\leq i$. W.l.o.g. assume $r'\leq r$. If $r=i$ we are done as $P_{d,i}(t)P_{d',i}(t)\in B_{d+d'+1,i}$, so assume $r<i$.
Apply (\ref{eq:poly}) to the pair of terms of degrees $a=r$ and $b=r'$ in $P_{d,i}(t)P_{d',i}(t)$. Repeat this process in each summand w.r.t. the pair of terms of degrees $a$ and $b$ such that $0<b\leq a\leq i-1$, as long as such a pair exists (indeed, there is at most one such pair in each summand). The end result
is that $P_{d,i}(t)P_{d',i}(t)$ is expressed as a nonnegative integer linear combination of polynomials in $B_{d+d'+1,i}$.
\end{proof}

Note that for $\mathcal{HS}(i,d)$ with $d\leq 4$ and $(i,d)\neq (2,4)$, Conjecture \ref{conj:missing-g-conj} holds. As mentioned before, for the case $i=1$ see Gal \cite{Gal}, and for the other cases use Theorem \ref{thm:stacked,Kalai} (LBT).
The new inequality for $\Delta\in \mathcal{HS}(2,4)$ reads as $\f_1\geq 6\f_0 - 21$. Note that Theorem \ref{thm:stacked,Kalai} gives $\f_1\geq 5\f_0 - 15$. For flag $4$-spheres $\f_1\geq 7\f_0 - 30$ holds.
\begin{lemma}
If $\Delta\in \mathcal{HS}(2,4)$ then $\f_1(\Delta)\geq 5\frac{6}{91}\f_0(\Delta) - 15$.
\end{lemma}
\begin{proof}[Sketch of proof]
By Theorem \ref{thm:missingLBT}(b) $\f_0(\Delta)\geq 8$, hence $6\f_0(\Delta) - 21 > 5\frac{6}{91}\f_0(\Delta) - 15$.
Assume that $\f_1(\Delta)< 6\f_0(\Delta) - 21$. Look on the $1$-skeleton of $\Delta$. In this graph the average vertex degree is $<12$, hence at least $\frac{6}{7}$ of the vertices have degree $<13$. Denote by $X$ the subset of vertices of degree $<13$. Then $X$ contains an independent set $Y$ of size $\geq \frac{6}{91}\f_0(\Delta)$.
We now use generic rigidity arguments for graphs.
The link of any vertex is in $\mathcal{HS}(2,3)$, hence is not stack, hence its graph has a generic $4$-stress. Thus the $1$-skeleton of the closed star of any vertex $v$ has a generic $5$-stress containing an edge with $v$ in its support.
Picking such stress for each $v\in Y$ yield an independent set of stresses, thus $\g_2(\Delta)\geq |Y|$, hence $\f_1(\Delta)\geq 5\frac{6}{91}\f_0(\Delta) - 15$.
\end{proof}
To end, the `first' new inequality to consider is the following.
\begin{problem}
Show that if $\Delta$ is the boundary of a simplicial $5$-polytope with no missing faces of dimension $>2$ then $\f_1(\Delta)\geq 6\f_0(\Delta) - 21$.
\end{problem}
\begin{acknowledgement}
I would like to thank Lou Billera, Gil Kalai, Etienne Rassart and Ed Swartz for helpful discussions. I would like to thank Christos Athanasiadis for pointing out a mistake in the calculations in Proposition \ref{prop:MPWclique}. Further thanks go to the referees, whose comments helped to improve the presentation.
\end{acknowledgement}
\bibliographystyle{amsplain}
\bibliography{biblio}

\providecommand{\bysame}{\leavevmode\hbox to3em{\hrulefill}\thinspace}
\providecommand{\MR}{\relax\ifhmode\unskip\space\fi MR }
\providecommand{\MRhref}[2]{%
  \href{http://www.ams.org/mathscinet-getitem?mr=#1}{#2}
}
\providecommand{\href}[2]{#2}
\begin{thebibliography}{10}

\bibitem{Athanasiadis}
Christos~A. Athanasiadis, \emph{Some combinatorial properties of flag
  simplicial pseudomanifolds and spheres}, preprint, arXiv:0807.4369 (2008).

\bibitem{BarnetteManifolds}
David Barnette, \emph{Graph theorems for manifolds}, Israel J. Math.
  \textbf{16} (1973), 62--72. \MR{MR0360364 (50 \#12814)}

\bibitem{Barnette-LBT}
\bysame, \emph{A proof of the lower bound conjecture for convex polytopes},
  Pacific J. Math. \textbf{46} (1973), 349--354. \MR{MR0328773 (48 \#7115)}

\bibitem{Billera-Lee}
Louis~J. Billera and Carl~W. Lee, \emph{A proof of the sufficiency of
  {M}c{M}ullen's conditions for {$f$}-vectors of simplicial convex polytopes},
  J. Combin. Theory Ser. A \textbf{31} (1981), no.~3, 237--255. \MR{MR635368
  (82m:52006)}

\bibitem{Bjorner-NoteAlexanderDuality}
Anders Bj{\"o}rner, Lynne~M. Butler, and Andrey~O. Matveev, \emph{Note on a
  combinatorial application of {A}lexander duality}, J. Combin. Theory Ser. A
  \textbf{80} (1997), no.~1, 163--165. \MR{MR1472111 (98h:05181)}

\bibitem{Charney-Davis}
Ruth Charney and Michael Davis, \emph{The {E}uler characteristic of a
  nonpositively curved, piecewise {E}uclidean manifold}, Pacific J. Math.
  \textbf{171} (1995), no.~1, 117--137. \MR{MR1362980 (96k:53066)}

\bibitem{Davis-Okun}
Michael~W. Davis and Boris Okun, \emph{Vanishing theorems and conjectures for
  the {$\ell\sp 2$}-homology of right-angled {C}oxeter groups}, Geom. Topol.
  \textbf{5} (2001), 7--74 (electronic). \MR{MR1812434 (2002e:58039)}

\bibitem{Gal}
{\'S}wiatos{\l}aw~R. Gal, \emph{Real root conjecture fails for five- and
  higher-dimensional spheres}, Discrete Comput. Geom. \textbf{34} (2005),
  no.~2, 269--284. \MR{MR2155722 (2006c:52019)}

\bibitem{Gromov}
Mikhael Gromov, \emph{Partial differential relations}, Ergebnisse der
  Mathematik und ihrer Grenzgebiete (3) [Results in Mathematics and Related
  Areas (3)], vol.~9, Springer-Verlag, Berlin, 1986. \MR{MR864505 (90a:58201)}

\bibitem{Kalai-LBT}
Gil Kalai, \emph{Rigidity and the lower bound theorem. {I}}, Invent. Math.
  \textbf{88} (1987), no.~1, 125--151. \MR{MR877009 (88b:52014)}

\bibitem{Kalai-manyspheres}
\bysame, \emph{Many triangulated spheres}, Discrete Comput. Geom. \textbf{3}
  (1988), no.~1, 1--14. \MR{MR918176 (89b:52025)}

\bibitem{McMullen-g-conj}
P.~McMullen, \emph{The numbers of faces of simplicial polytopes}, Israel J.
  Math. \textbf{9} (1971), 559--570. \MR{MR0278183 (43 \#3914)}

\bibitem{Meshulam-DomAndHom}
Roy Meshulam, \emph{Domination numbers and homology}, J. Combin. Theory Ser. A
  \textbf{102} (2003), no.~2, 321--330. \MR{MR1979537 (2004c:05144)}

\bibitem{Nagel}
Uwe Nagel, \emph{Empty simplices of polytopes and graded {B}etti numbers},
  Discrete Comput. Geom. \textbf{39} (2008), no.~1-3, 389--410. \MR{MR2383766}

\bibitem{Nevo-Novinsky}
Eran Nevo and Eyal Novinsky, \emph{A characterization of simplicial polytopes
  with $g_2=1$}, preprint, arXiv:0804.1813 (2008).

\bibitem{Stanley-gThm}
Richard~P. Stanley, \emph{The number of faces of a simplicial convex polytope},
  Adv. in Math. \textbf{35} (1980), no.~3, 236--238. \MR{MR563925 (81f:52014)}

\bibitem{StanleyGreenBook}
\bysame, \emph{Combinatorics and commutative algebra}, second ed., Progress in
  Mathematics, vol.~41, Birkh\"auser Boston Inc., Boston, MA, 1996.
  \MR{MR1453579 (98h:05001)}

\end{thebibliography}
\end{document}